\documentclass[letterpaper, 10 pt, conference]{ieeeconf}  

\IEEEoverridecommandlockouts                              

\usepackage{ntheorem}
\usepackage{hyperref}
\usepackage{textcomp}
\let\labelindent\relax
\usepackage{enumitem}
\usepackage{graphicx}
\usepackage{epstopdf}
\usepackage{float}
\usepackage[nocomma]{optidef}
\usepackage{dsfont}
\usepackage{times}
\usepackage{bm}
\usepackage{fancyhdr,graphicx,amsmath,amssymb}
\usepackage{algorithm} 
\usepackage{algpseudocode} 
\include{pythonlisting}
\usepackage{cite}
\usepackage{xcolor}
\usepackage{amssymb}
\usepackage{mathtools}

\makeatletter
\def\BState{\State\hskip-\ALG@thistlm}
\makeatother
\usepackage{stfloats}

\title{\LARGE \bf
On Robustness of the Normalized Subgradient Method with\\ Randomly Corrupted Subgradients
}

\author{{Berkay Turan}\and
{C\'esar A. Uribe}\and
{Hoi-To Wai}
\and
{Mahnoosh Alizadeh}
\thanks{B. Turan and M. Alizadeh are with  Dept. of ECE, UCSB, Santa Barbara, CA, USA. C. A. Uribe is with Dept. of ECE, Rice University, TX, USA. H. T. Wai is with Dept. of SEEM, CUHK, Shatin, Hong Kong. This work is supported by UCOP Grant LFR-18-548175, NSF grant \#1847096, and CUHK Direct Grant \#4055113. E-mails: \url{bturan@ucsb.edu}, \url{cauribe@rice.edu}, \url{htwai@se.cuhk.edu.hk}, \url{alizadeh@ucsb.edu}}
}

\begin{document}
\addtolength{\abovedisplayskip}{-.05cm}
\addtolength{\belowdisplayskip}{-.05cm}
\addtolength{\textfloatsep}{-.5cm}

\maketitle
\thispagestyle{empty}
\pagestyle{empty}

\begin{abstract}
Numerous modern optimization and machine learning algorithms rely on subgradient information being trustworthy and hence, they may fail to converge when such information is corrupted. In this paper, we consider the setting where subgradient information may be arbitrarily corrupted (with a given probability) and study the robustness properties of the normalized subgradient method. Under the probabilistic corruption scenario, we prove that the normalized subgradient method, whose updates rely solely on  directional information of the subgradient, converges to a minimizer for convex, strongly convex, and weakly-pseudo convex functions satisfying certain conditions. 
Numerical evidence on linear regression and logistic classification problems supports our results.

\end{abstract}
\theoremseparator{.}

\newtheorem{proposition}{Proposition}
\newtheorem{theorem}{Theorem}
\newtheorem*{theoremnonumber}{Theorem}
\newtheorem{corollary}{Corollary}
\newtheorem{lemma}{Lemma}
\newtheorem{Fact}{Fact}
\newtheorem{remark}{Remark}
\newtheorem{assumption}{Assumption}
\newtheorem{definition}{Definition}

\newcommand{\eqdef}{\vcentcolon=}
\newcommand{\beq}{\begin{equation}}
\newcommand{\eeq}{\end{equation}}
\newcommand{\ie}{i.e., }

\section{INTRODUCTION}
Gradient-based methods are  the most commonly used optimization algorithms in many modern applications including machine learning, control, and signal processing. Owing to their convenience for theoretical analysis and implementation, extensive research focusing on accelerating gradient descent has emerged in the literature. Existing algorithms aim to accelerate gradient descent by choosing adaptive stepsizes \cite{qian1999momentum,nesterov1983method,duchi2011adaptive,zeiler2012adadelta,tieleman2012lecture,kingma2014adam,dozat2016incorporating}. In addition to their practical success, these methods are supported by  strong  theoretical  guarantees  such as fast convergence rates.  Furthermore,  distributed implementations of optimization algorithms, where the data is stored in worker machines, are gaining traction to preserve privacy. In these schemes, the gradients are computed  at the worker machines and communicated to a master machine \cite{yang2019federated}.

However, the above algorithms rely on the assumption that the gradients are \emph{trustworthy}. In reality, the gradients might be erroneous due to computational errors at the machines or data corruption caused by an adversary. Moreover,  distributed implementations of these algorithms require reliable communication between the machines. Such schemes are susceptible to man-in-the-middle attacks, where an adversary can  take  over  network sub-systems and arbitrarily alter the information communicated  between the machines to prevent convergence to the optimal solution, \ie a Byzantine attack \cite{lamport2019byzantine}. In these situations, gradient-based methods that use the exact gradient information may fail to converge, as the erroneous gradient can have an arbitrarily large effect on the algorithm. 

This paper studies the robustness property of the normalized subgradient method (NSM) \cite{nesterov2004introductory}, which performs a subgradient update with an adaptive step size scaled as the reciprocal of the norm of subgradient. We consider a probabilistic corruption scenario where the subgradients received are arbitrarily corrupted at random. As the NSM only uses the directional information of the subgradient, the effects of corrupted subgradients will be limited. We show that when the subgradient corruption probability is below a certain threshold, the NSM converges to an optimal solution for convex and weakly pseudo-convex problems. In particular, let $t$ be the iteration number and let the stepsize decay at ${\cal O}(1/t)$, then the NSM finds an optimal solution at the rate of ${\cal O}(\log(t)/t)$.  


\noindent
\textbf{Related Work:} There is an extensive literature  on developing and analyzing the convergence of optimization algorithms, including those that consider adversarial manipulation.  Limited by space, we review categories of prior art that we believe are most related to this work.
\begin{enumerate}[wide, labelindent=0pt]
\item \emph{Normalized (sub)gradient method:} NSM is a well-studied algorithm for optimization and is supported by theoretical convergence guarantees \cite{shor2012minimization,nesterov2004introductory}. When restricted to optimizing differentiable functions, the NSM has been shown to evade saddle points quickly in non-convex optimization problems \cite{levy2016power,murray2019revisiting}. Moreover, \cite{hazan2015beyond} studies a stochastic version of NSM and shows convergence to a global minimum for quasi-convex functions. A recent work \cite{gao2018online} studies an online learning problem and establishes regret bounds for normalized gradient methods for weakly pseudo-convex functions. Scholars have also used normalization of the gradient in first-order optimization algorithms to accelerate the convergence \cite{fang2018spider,curtis2019stochastic}.
\item \emph{Learning under adversarial corruption:} The task of learning under adversarial corruption has been a popular research topic. The authors of \cite{huang2011adversarial} provide a comprehensive survey on effective learning techniques against an adversarial opponent and give a taxonomy for classifying attacks. Commonly studied tasks for adversarial machine learning have been classification \cite{kurakin2016adversarial,lowd2005adversarial,jang2017objective,dalvi2004adversarial} and linear regression \cite{chen2013robust,bhatia2015robust,mcwilliams2014fast,diakonikolas2019sever}, where the adversarial corruption is usually due to data manipulation. Closest to our setup would be \cite{mahloujifar2017blockwise,mahloujifar2018learning}, which consider a probabilistic corruption scenario  called $p$-tampering. However, in these works the adversary is restricted to choosing valid tampered data with correct labels instead of the arbitrary corruption we study in this paper.
\item \emph{Robust distributed optimization: }This line of work \cite{turan2020resilient,ghosh2019communication,pillutla2019robust,wu2020federated,ghosh2019robust,yin2018byzantine,yin2019defending,chen2017distributed,bernstein2018signsgd} focuses on developing robust algorithms for distributed optimization/federated learning by means of robust aggregation techniques. In these studies, there is assumed to be a constant $\alpha$ fraction of Byzantine worker machines that are being controlled by an adversary, who alters the messages transmitted from workers to the master machine in order to prevent convergence of distributed implementation of optimization algorithms. The popular solution idea is to filter out the adversarial messages by robust mean estimation of the messages received from all the worker machines.
\end{enumerate}
There is no existing work on robust optimization in a probabilistically and arbitrarily corrupted subgradient setting to the best of our knowledge. In this paper, we study the robustness of the NSM in this context.

\noindent
\textbf{Paper Organization:}
The remainder of the paper is organized as follows. In Section~\ref{sec:problem}, we formalize the problem setup and demonstrate how faulty subgradients can affect the convergence. In Section~\ref{sec:ngm}, we describe the normalized subgradient method (Algorithm~\ref{alg:normalizedgd}) and analyze its convergence in a randomly corrupted subgradient setting for convex, strongly convex, and weakly pseudo-convex functions. In Section~\ref{sec:numerical}, we provide numerical experiments demonstrating the robustness of the normalized subgradient method.

\noindent
\textbf{Notations}. Unless otherwise specified, $\| \cdot \|$ denotes the standard Euclidean norm. Given a vector $x_t$, $x_{t,i}$ indicates the $i$'th entry of $x_t$. The abbreviation \textit{a.s.} indicates almost sure convergence.

\section{PROBLEM SETUP}\label{sec:problem}
In this section, we formally set up our problem and introduce key concepts and definitions that will be used in the paper. Let $f(\cdot):\mathbb{R}^d\rightarrow\mathbb{R}$ be a cost function of a parameter vector $x\in {\cal X}\subseteq \mathbb{R}^d$, where $\cal X$ is the parameter set that is assumed to be convex and compact with diameter $R$, \ie$\|x_1-x_2\|\leq R,\forall x_1,x_2\in{\cal X}$. Our goal is to find the parameter that minimizes the cost function:
\begin{equation}\label{eq:optimizationproblem}
    x^\star=\underset{x\in\cal X}{\arg \min}~f(x).
\end{equation}
(Sub)gradient-based methods are effective for iteratively solving \eqref{eq:optimizationproblem}. At each iteration $t$, the iterate $x_t$ is updated upon observing the feedback $g(x_t)$, where $g(x_t)\in \partial f(x_t)$ such that $\partial f(x_t)$ is the subdifferential of $f$ at $x_t$; note that if $f$ is differentiable, then $\partial f(x_t) = \{ \nabla f(x_t) \}$ (We use the notation $g(x_t)=\nabla f(x_t)$ even if $f$ is differentiable but non-convex). 
For simplicity of exposition, we will use subdifferential and gradient interchangeably depending on whether the function is differentiable or not. 

Gradient-based methods however rely on the feedback received being a trustworthy subgradient information. In this paper, we consider the case where at each iteration $t$, the feedback received is corrupt with probability $p$, potentially due to an adversarial attack. Therefore at each iteration $t$, the feedback is determined as:
\begin{equation}\label{eq:gradients}
    h_t=\left\{\begin{array}{cl}
     g(x_t) & \text{with probability } 1-p, \\
        b_t & \text{with probability } p.
    \end{array}\right.
\end{equation}
We let the corrupt feedback $b_t$ to be chosen potentially by an adversary, who is assumed to have full knowledge of the problem. The adversary's goal is to prevent the iterates $x_t$ from converging to the optimal solution. We note that this model encompasses all the cases where the feedback can become corrupt (e.g., communication errors, computational errors, noise) since we set no restrictions on $b_t$. It is important to mention that because $b_t$ is assumed to be arbitrary, existing subgradient methods under stochastic errors, e.g., \cite{boyd2008stochastic,borkar2009stochastic,bertsekas2000gradient}, are not applicable in this setting. This literature models the error on the gradients as additive noise with bounded mean and the variance, whereas the adversarial manipulation we consider in this work can be unbounded and arbitrary.

\noindent
\textbf{A toy example:} In order to make the problem setup clearer and to demonstrate how corrupt subgradients can affect the convergence, let us study a simple example. Let $x\in {\cal X}\subset \mathbb{R}^d$, where ${\cal X}=\{{ x}|x_0=x_1=\dots =x_{d-1}; |x_i|\leq R\}$, and $f(x)=x_{d-1}^4$. The optimal solution is $x^\star=[0,\cdots,0]^T$. The function $f$ is differentiable, the gradient at iteration $t$ is:
\begin{equation}
    g(x_t)=[0,~0,~\dots,~0,~4x_{t,0}^3]^T.
\end{equation}
Suppose that $x_{t,0}>0$ and we are applying the regular gradient descent algorithm. Given stepsize $\gamma_t$, if $h_t=g(x_t)$, the next iterate after projection is $x_{t+1}=[x_{t,0}-{4x_{t,0}^3\gamma_t}/{d},\dots,~x_{t,0}-{4x_{t,0}^3\gamma_t}/{d}]^T$. Suppose that at each iteration $t$, the adversary has full knowledge and sets the corrupt gradient as
\begin{equation}
    b_t=[-c,~-c,~-c,~\dots,~-c,~-c]^T,
\end{equation}
and therefore if $h_t=b_t$, the next iterate is
$x_{t+1}=\min\{[R,\dots,~R]^T,[x_{t,0}+\gamma_tc,~\dots,~x_{t,0}+\gamma_tc]^T\}$. We can observe that if the adversary sets $c>({1/p-1})({4R^3}/{d})$, then $x_{t+1,i}\geq x_{t,i}$ in expectation and the adversary will prevent convergence (since $x_t\leq R$).

\begin{algorithm}[t]
    \caption{Normalized Subgradient Method (NSM)}
    \begin{algorithmic}
    \Require Initialize $x_1\in {\cal X}$, step size $\gamma_t$, and $T$    
        \For{$t=1$ to $T$}
            \State Given $x_t$, receive feedback $h_t$ according to \eqref{eq:gradients}.
            \If{$\|h_t\|>0$}
            \State $x_{t+1}=\Pi_{\cal X}\left(x_t-\gamma_t\frac{h_t}{\|h_t\|}\right)$
            \Else
            \State $x_{t+1}=x_t$
            \EndIf
        \EndFor
    \end{algorithmic}
    \label{alg:normalizedgd}
\end{algorithm}

Let us study how the normalized subgradient method outlined in Algorithm~\ref{alg:normalizedgd} would potentially solve this problem. At iteration $t$, the normalized gradient is:
\begin{equation}
    \frac{g(x_t)}{\|g(x_t)\|}=[0,~0,~\dots,~0,~1]^T.
\end{equation}
The normalized corrupt gradient is:
\begin{equation}
    \frac{b_t}{\|b_t\|}=\frac{[-1,~-1,~-1,~\dots,~-1,~-1]^T}{\sqrt{d}}
\end{equation}
If $h_t=g(x_t)$,  $x_{t+1}=[x_{t,0}-{\gamma_t}/{d},\dots,~x_{t,0}-{\gamma_t}/{d}]^T$. If $h_t=b_t$, $ x_{t+1}=[x_{t,0}+{\gamma_t}/{\sqrt{d}},~\dots,~x_{t,0}+{\gamma_t}/{\sqrt{d}}]$ assuming that $x_{t,0}+{\gamma_t}/{\sqrt{d}}\leq R$. Therefore, in expectation:
\begin{equation}
     \mathbb{E}[x_{t+1}|x_t]=[x_{t,0}-(1-p){\gamma_t}/{d}+p{\gamma_t}/{\sqrt{d}},~\dots]^T.
\end{equation}
Hence, if $p<{1}/({1+\sqrt{d}})$, the iterates will converge; if $p>{1}/({1+\sqrt{d}})$ they will diverge in expectation. This is not the case for the regular gradient descent, where the adversary can set the magnitude of the corrupt gradient large enough to prevent convergence even for low $p$. Furthermore, the threshold probability $p$, below which the normalized subgradient method succeeds, is an inherent feature of the cost function that we will discuss later in Section~\ref{sec:prob}.

In the next section, we will analyze the convergence of the NSM in the previously described randomly corrupted subgradient setting for (a) convex, (b) strongly convex, and (c) weakly pseuodo-convex functions. We conclude this section defining the class of functions studied in this paper.
\begin{definition}\label{def:smooth}
A differentiable function $f(\cdot)$ is said to be \textbf{$\boldsymbol{\beta}$-smooth} if there exists $\beta>0$ such that
\begin{equation}
    \|g(x_1)-g(x_2)\|\leq \beta \|x_1-x_2\|
\end{equation}
holds for all $x_1,x_2\in \cal X$.
\end{definition}
\begin{definition}\label{def:strong}
A differentiable function $f(\cdot)$ is said to be \textbf{$\boldsymbol{\mu}$-strongly convex} if there exists $\mu>0$ such that
\begin{equation}
    \langle g(x_1)-g(x_2),x_1-x_2 \rangle\geq \mu\|x_1-x_2\|^2
\end{equation}
holds for all $x_1,x_2\in \cal X$.
\end{definition}
\begin{definition}\label{def:acute}
A function $f(\cdot)$ is said to satisfy the \textbf{acute angle} condition if there exists some $\phi$ satisfying $0\leq \phi<\frac{\pi}{2}$ such that
\begin{equation}\label{eq:acutangle}
    \langle g(x),x-x^\star(x)\rangle\geq \cos{\phi}\|g(x)\|\|x-x^\star(x)\|
\end{equation}
holds for all $x\in {\cal X}$ and $g(x)\in \partial f(x)$, where $x^\star(x)$ is the point in the set of minima of $f$ that is nearest to $x$.
\end{definition}
As an example, the cost function studied in the simple example $f(x)=x_{d-1}^4$ satisfies the acute angle condition with $\cos{\phi}={1}/{\sqrt{d}}$. Another example is $f(x)=|x_1|+|x_2|$, which satisfies the acute angle condition with $\cos{\phi}=1/\sqrt{2}$.
\begin{definition}\label{def:wpc}
A differentiable function $f(\cdot)$ is said to be \textbf{weakly pseudo-convex} if there exists $K>0$ such that
\begin{equation}
    f(x)-f(x^\star)\leq K\frac{\langle g(x),x-x^\star(x)\rangle}{\|g(x)\|}
\end{equation}
holds for all $x\in{\cal X}$, with the convention that ${g(x)}/{\|g(x)\|}=0$ if $g(x)=0$, where $x^\star(x)$ is the point in the set of minima of $f$ that is nearest to $x$.
\end{definition}
Some examples of weakly pseudo-convex functions are:
\begin{itemize}[leftmargin=4mm]
    \item Differentiable, Lipschitz continuous, and pseudo-convex functions \cite{nesterov2004introductory};
    \item Star-convex \cite{nesterov2006cubic} and smooth functions (e.g., $f(x)=|x|(1-e^{|x|})$);
    \item Functions with bounded gradient that satisfy the acute angle condition;
    \item Functions with bounded gradient that satisfy the $\alpha$-homogeneity with respect to its minimum; (\cite{gao2018online}, Proposition 3), \ie there exists $\alpha>0$ such that
    \begin{equation}
        f(t(x-x^\star)+x^\star)-f(x^\star)=t^\alpha(f(x)-x(x^\star))
    \end{equation}
    holds for all $x\in {\cal X}$ and $t\geq 0$. For instance, $f(x)=(x_1^2+x_2^2)^2+10(x_1^2-x_2^2)^2$ satisfies this condition and is not convex.
\end{itemize}
\section{ROBUSTNESS ANALYSIS}\label{sec:ngm}
In this section, we study the normalized subgradient method and show that it can be used to solve the problem defined in Section~\ref{sec:problem}. The idea behind is that by normalization, the feedback is restricted to contain only a directional information. This allows us to limit the adversary's attack potential by not allowing arbitrarily large updates, which would have been possible without normalization.

We summarize the normalized subgradient method in Algorithm~\ref{alg:normalizedgd}. At each iteration $t$, the algorithm uses the feedback $h_t$ to compute the normalized vector $h_t/\|h_t\|$ as the update direction. Similar to a standard gradient descent method, it moves the iterate along that direction with stepsize $\gamma_t$ and then projects the point back to the decision set $\cal X$.

In order to prove the convergence of Algorithm~\ref{alg:normalizedgd} in a randomly corrupted subgradient setting, we will use the Robbins-Siegmund Theorem \cite{robbins1971convergence} as an auxiliary result:

\begin{theorem}[Robbins-Siegmund]\label{thm:RS}
Let $(V_t)_{t\geq 1}$, $(\alpha_t)_{t\geq 1}$, $(\chi_t)_{t\geq 1}$, $(\eta_t)_{t\geq 1}$ be four nonnegative $({\cal F})_{t\geq 1}$-adapted processes such that $\sum_t \alpha_t<\infty$ and $\sum_t \chi_t<\infty$ almost surely. If for each $t\in \mathbb{N}$,
\begin{equation}
    \mathbb{E}\left[V_{t+1}|{\cal F}_t\right]\leq V_t(1+\alpha_t)+\chi_t-\eta_t
\end{equation}
then $(V_t)_{t\geq 1}$ converges almost surely to a random variable $V_\infty$ and $\sum_t \eta_t$ is finite almost surely.
\end{theorem}

We can now state the main result on the convergence of the normalized subgradient method for convex (but possibly non-smooth) functions that meet the acute angle condition (\ie functions that meet Definition~\ref{def:acute}) in a randomly corrupted subgradient setting.  

\begin{theorem}\label{thm:main}
Let $f(\cdot)$ be a convex function defined on $\cal X$ satisfying the acute angle condition for some $\phi\in[0,\pi/2)$. 
Suppose that the corruption probability $p$ satisfies
\begin{equation*}\label{eq:probupperbound}
    p<\frac{\cos{\phi}}{1+\cos{\phi}}.
\end{equation*}
Let $\gamma =\frac{R}{2((1-q)\cos{\phi}-q)}$ for some $q\in[p,\frac{\cos{\phi}}{1+\cos{\phi}})$ and set the step size as $\gamma_t = \gamma / t$. Then, the iterates generated by Algorithm~\ref{alg:normalizedgd} have the following properties:
\begin{equation}\label{eq:thmresult1}
    \mathbb{E}\left[\|x_{T+1}-x^\star(x_{T+1})\|^2\right]\leq\frac{\gamma^2(1+\log{T})}{T}
\end{equation}
and
\begin{equation}\label{eq:thmresult2}
 \underset{t\rightarrow\infty}{\lim}\|x_t-x^\star(x_t)\|^2=0,\quad a.s.
\end{equation}
\end{theorem}
\begin{proof}
Let $Y_t$ be a Bernoulli($p$) random variable, which indicates whether the subgradient is trustworthy or corrupt:
\begin{align}
    &\|x_{t+1}-x^\star(x_{t+1})\|^2\overset{(a)}{\leq} \|x_{t+1}-x^\star(x_t)\|^2 \\
    &=\|{\Pi}_{\cal X}(x_t-\gamma_t \frac{h_t}{\|h_t\|})-x^\star(x_t)\|^2\\
    &\overset{(b)}{\leq}\|x_t-x^\star(x_t)-\gamma_t \frac{h_t}{\|h_t\|}\|^2\\
    &=\|x_t-x^\star(x_t)\|^2+\gamma_t^2-2\gamma_t\langle x_t-x^\star(x_t),\frac{h_t}{\|h_t\|}\rangle\\
    \begin{split}\label{eq:20}
        &=\|x_t-x^\star(x_t)\|^2+\gamma_t^2-2\gamma_t Y_t\langle x_t-x^\star(x_t),\frac{b_t}{\|b_t\|}\rangle\\
        &\quad-2\gamma_t (1-Y_t)\langle x_t-x^\star(x_t),\frac{g(x_t)}{\|g(x_t)\|}\rangle
    \end{split}\\
     \begin{split}\label{eq:21}
        &\overset{(c)}{\leq}\|x_t-x^\star(x_t)\|^2+\gamma_t^2+2\gamma_t Y_t\|x_t-x^\star(x_t)\|\\
        &\quad-2\gamma_t (1-Y_t)\langle x_t-x^\star(x_t),\frac{g(x_t)}{\|g(x_t)\|}\rangle
    \end{split}\\
    \begin{split}
        &\overset{(d)}{\leq}\|x_t-x^\star(x_t)\|^2+\gamma_t^2+2\gamma_t Y_t\|x_t-x^\star(x_t)\|\\
        &\quad-\frac{2\gamma_t(1-Y_t)}{\|g(x_t)\|}\cos{\phi} \|x_t-x^\star(x_t)\|\|g(x_t)\|
    \end{split}\\
        \begin{split}
        &=\|x_t-x^\star(x_t)\|^2+\gamma_t^2\\&\quad - 2\gamma_t\|x_t-x^\star(x_t)\|(\cos{\phi}(1-Y_t)-Y_t),
    \end{split}
    \label{eq:pfbeforeexpectation}
\end{align}
where (a) follows by definition of $x^\star(x)$, (b) is due to nonexpansiveness Euclidean projection, (c) is due to Cauchy–Schwarz inequality, and (d) uses \eqref{eq:acutangle}. Taking expectation of both sides and noting that $x_t$ and $Y_t$ are independent:
\begin{equation}\label{eq:pfafterexpectation}
    \begin{split}
        &\mathbb{E}\left[\|x_{t+1}-x^\star(x_{t+1})\|^2\right]\leq\mathbb{E}\left[x_t-x^\star(x_t)\|^2\right]+\gamma_t^2\\
        &\hspace{1cm}-2\gamma_t(\cos{\phi}(1-p)- p)\mathbb{E}\left[\|x_t-x^\star(x_t)\|\right].
    \end{split}
\end{equation}
Since $p<\frac{\cos{\phi}}{1+\cos{\phi}}$; $(\cos{\phi}(1-p)-p)>0$. Therefore, to upper bound the inequality, we lower bound $\mathbb{E}\left[\|x_t-x^\star(x_t)\|\right]$ as:
\begin{equation}\label{eq:25}
\begin{split}
    \mathbb{E}\left[\|x_t-x^\star(x_t)\|\right]&=\mathbb{E}\left[\frac{\|x_t-x^\star(x_t)\|^2}{\|x_t-x^\star(x_t)\|}\right]\\
    &\geq \frac{\mathbb{E}\left[\|x_t-x^\star(x_t)\|^2\right]}{R}
\end{split}
\end{equation}
Therefore \eqref{eq:pfafterexpectation} becomes
\begin{equation}
    \begin{split}
           &\mathbb{E}\left[\|x_{t+1}-x^\star(x_{t+1})\|^2\right]\leq \gamma_t^2\\
           &+\mathbb{E}\left[x_t-x^\star(x_t)\|^2\right]\left(1-\gamma_t\frac{2(\cos{\phi}(1-p)- p)}{R}\right).
    \end{split}
\end{equation}
Plugging $\gamma_t$, we obtain:
\begin{align}
    \begin{split}
         &\mathbb{E}\left[\|x_{t+1}-x^\star(x_{t+1})\|^2\right]\leq \frac{\gamma^2}{t^2}\\
           &+\mathbb{E}\left[x_t-x^\star(x_t)\|^2\right]\left(1-\frac{(\cos{\phi}(1-p)- p)}{(\cos{\phi}(1-q)- q)}\frac{1}{t}\right)
    \end{split}\\
    &\leq \frac{\gamma^2}{t^2}+\mathbb{E}\left[x_t-x^\star(x_t)\|^2\right]\left(1-\frac{1}{t}\right),
\end{align}
where the last inequality uses $\frac{\cos{\phi}(1-p)- p}{\cos{\phi}(1-q)-q}\geq 1$ for $q\in[p,\frac{\cos{\phi}}{1+\cos{\phi}})$. Finally, a telescopic sum gives:
\begin{align}
    \begin{split}
        \mathbb{E}\left[\|x_{T+1}-x^\star(x_{T+1})\|^2\right] &\leq\gamma^2\sum_{t=1}^T\frac{1}{t^2}\prod_{i=t+1}^T(1-\frac{1}{i})\\
         &\hspace{-1cm}+\|x_1-x^\star(x_1)\|^2\prod_{t=1}^T(1-\frac{1}{t})
     \end{split}\\
     & \hspace{-3cm}=\gamma^2\sum_{t=1}^T\frac{1}{t^2}\frac{t}{T}=\frac{\gamma^2}{T}\sum_{t=1}^T\frac{1}{t}\leq\frac{\gamma^2(1+\log{T})}{T}\label{eq:pfthm1result1},
\end{align}
which completes the first part of the proof. To prove the almost sure convergence of $\|x_t-x^\star(x_t)\|^2$, we go back to~\eqref{eq:pfbeforeexpectation} and take the expectation conditioned on $x_t$ to obtain:
\begin{align}
\begin{split}
      \mathbb{E}[\|x_{t+1}-x^\star(x_{t+1})\|^2|x_t]&\leq\|x_t-x^\star(x_t)\|^2+\gamma_t^2\\
      &\hspace{-2cm}-2\gamma_t\frac{\cos{\phi}(1-p)-p}{ R}\|x_t-x^\star(x_t)\|^2
\end{split}
  \\
    &\hspace{-3.6cm}=\|x_t-x^\star(x_t)\|^2\left(1-2\gamma_t\frac{\cos{\phi}(1-p)- p}{ R}\right)+\gamma_t^2\\
    &\hspace{-3.6cm}\leq\|x_t-x^\star(x_t)\|^2\left(1-\frac{1}{t}\right)+\frac{\gamma^2}{t^2}
\end{align}

We now apply Theorem~\ref{thm:RS} with $V_t=\|x_t-x^\star(x_t)\|^2$, $\alpha_t=0$, $\eta_t=\|x_t-x^\star(x_t)\|^2/t$, and $\chi_t={\gamma^2}/{t^2}$ to conclude that $\|x_t-x^\star(x_t)\|^2$ converges almost surely and $\sum_t \|x_t-x^\star(x_t)\|^2/t$ is finite almost surely. In order to determine where $\|x_t-x^\star(x_t)\|^2$ converges as well, we use~\eqref{eq:pfthm1result1} to obtain:
\begin{equation}
    \underset{t\rightarrow\infty}{\lim}\mathbb{E}\left[\|x_{t}-x^\star(x_t)\|^2\right]=0.
\end{equation}
Finally, since $\|x_{t}-x^\star(x_t)\|^2\geq 0$, $\underset{t\rightarrow \infty}{\lim}\|x_t-x^\star(x_t)\|^2=0$ almost surely and therefore $x_t$ converges to a minimum point~$x^\star$ almost surely.
\end{proof}

Theorem~\ref{thm:main} is a general result for convex functions satisfying the acute angle condition (\ie convex functions that meet Definition~\ref{def:acute}). The next corollary states a similar result for strongly convex and smooth functions (\ie functions that meet Definitions~\ref{def:smooth} and \ref{def:strong}).
\begin{corollary}\label{cor:strong}
Suppose that $f(\cdot)$ is $\mu$-strongly convex and $\beta$-smooth. Assume that the parameter set $\cal X$ contains a non-empty ball centered at $x^\star$. Define condition number of $f(\cdot)$ as $\kappa\eqdef{\beta}/{\mu}$. If the corruption probability satisfies $p<\frac{1}{1+\kappa}$,
then the iterates generated by Algorithm~\ref{alg:normalizedgd} with $\gamma=\frac{\kappa R}{2((1-q)-q\kappa)}$, $\gamma_t = \gamma / t$ for some $q\in[p,\frac{1}{1+\kappa})$ have the following properties:
\begin{equation}
    \mathbb{E}\left[\|x_{T+1}-x^\star\|^2\right]\leq\frac{\gamma^2(1+\log{T})}{T}
\end{equation}
and
\begin{equation}
 \underset{t\rightarrow\infty}{\lim}x_t=x^\star,\quad a.s.
\end{equation}
where $x^\star$ is the unique minimizer of $f(\cdot)$.
\end{corollary}
\begin{proof}
If the parameter set contains a non-empty ball centered at $x^\star$, then $g(x^\star)=0$. Using strong convexity and smoothness properties, we  write:
\begin{align}
    &\langle g(x),x-x^\star \rangle\geq \mu\|x-x^\star\|^2=\frac{\mu\|x-x^\star\|^2\|g(x)\|}{\|g(x)\|} \nonumber\\
    &\geq\frac{\mu\|x-x^\star\|^2\|g(x)\|}{\beta\|x-x^\star\|}=\frac{1}{\kappa}\|x-x^\star\|\|g(x)\|.
\end{align}
By setting $\cos{\phi}=\frac{1}{\kappa}$, we get the required condition in \eqref{eq:acutangle} and hence the results from Theorem~\ref{thm:main} follow.
\end{proof}

The class of functions for which Algorithm~\ref{alg:normalizedgd} provides convergence in a randomly corrupted subgradient setting go beyond convex and strongly convex functions. The next corollary is for weakly pseudo-convex functions satisfying the acute angle condition (\ie functions that meet Definitions~\ref{def:acute} and \ref{def:wpc}).
\begin{corollary}\label{cor:wpc}
Suppose that $f(\cdot)$ is weakly pseudo-convex and satisfies the acute angle condition for some $\phi\in[0,\pi/2)$. If the corruption probability $p$ satisfies $p<\frac{\cos{\phi}}{1+\cos{\phi}}$,
then the iterates generated by Algorithm~\ref{alg:normalizedgd} with $\gamma=\frac{R}{2((1-q)\cos{\phi}-q)}$, $\gamma_t = \gamma / t$ for some $q\in[p,\frac{\cos{\phi}}{1+\cos{\phi}})$ satisfy:
\begin{equation}
    \mathbb{E}\left[\|x_{T+1}-x^\star(x_{T+1})\|^2\right]\leq\frac{\gamma^2(1+\log{T})}{T}
\end{equation}
and
\begin{equation}
 \underset{t\rightarrow\infty}{\lim}\|x_t-x^\star(x_t)\|^2=0,\quad a.s.
\end{equation}
\end{corollary}
\begin{proof}
    By Definition~\ref{def:wpc}, if $f(\cdot)$ is weakly pseudo-convex, then it does not have local extrema nor saddle points. Therefore,  $g(x)=0$ only if $x\in \underset{x\in\cal X}{\arg \min}~f(x)$. The proof is then identical to that of Theorem~\ref{thm:main}.
\end{proof}
\begin{remark}
The class of non-convex smooth functions satisfying the acute angle condition also meet Corollary~\ref{cor:wpc}. If $f$ is smooth and the diameter of $\cal X$ is bounded, then $\|g(\cdot)\|$ is bounded. By Proposition 2 in \cite{gao2018online},  if $f$ has bounded gradient and satisfies the acute angle condition, then $f$ is weakly pseudo-convex. Therefore, Corollary~\ref{cor:wpc} holds for $f$.
\end{remark}

We would like to highlight the implication of the sufficient condition on the corruption probability, which sets an upper bound on $p$ for Theorem~\ref{thm:main} and Corollaries~\ref{cor:strong}~\&~\ref{cor:wpc}. Observe that the upper bound is a decreasing function of the angle $\phi$ of the acute angle condition (or $\kappa$, the condition number for strongly convex functions). The acute angle condition determines an upper bound on the angle between $g(x)$ and $x-x^\star(x)$, $\forall x\in{\cal X}$. Therefore, the smaller the angle between $g(x)$ and $x-x^\star(x)$, the more accurate the update direction, since the goal is to find $x^\star(x)$. This reflects to the upper bound on $p$: The functions that meet the acute angle condition with small $\phi$ (or small condition number $\kappa$, \ie well-conditioned problems) can tolerate a higher corruption probability and therefore NSM is more robust to adversarial corruption. The intuition behind is that because the worst adversary may know the optimal solution, they can set the feedback along the direction of $x^\star(x)-x$ (which meets the Cauchy-Schwarz inequality used for Equation~\eqref{eq:20} to \eqref{eq:21}). Therefore, the more aligned the direction of the trustworthy subgradient with $x-x^\star(x)$, the more effective the trustworthy updates and the NSM can tolerate higher corruption probabilities.

In the next section, we present the experiments on the robustness of the normalized subgradient method.
\section{NUMERICAL EXPERIMENTS}\label{sec:numerical}
\begin{figure}[t]
    \centering
    \includegraphics[width=.5\textwidth]{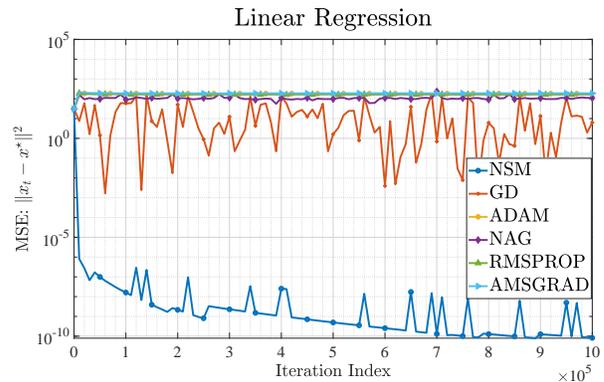}
    \vspace{-.5cm}
    \caption{Performances of first-order algorithms for Linear Regression using mean squared error as metric.}
    \label{fig:linearregression}
\end{figure}

In this section, we demonstrate the robustness of the normalized subgradient method in an adversarial setting by applying it to linear regression and logistic classification problems. Furthermore, we include performances of several other first-order algorithms for comparison. Lastly, we demonstrate the effect of the threshold corruption probability using the toy example discussed in Section~\ref{sec:problem}.
\subsection{Linear Regression}

Consider the following problem:
\begin{equation}
    \underset{x\in{\cal X}}{\min}~\|y-Ax\|^2,
\end{equation}
where $A\in\mathbb{R}^{N\times d}$ is a matrix containing the $N$ data vectors in its rows and $y\in\mathbb{R}^N$ is the vector containing the $N$ associated label values or outputs. The task is to minimize the empirical quadratic error between a function $y_i=g(A_i)$ of $d$ variables and its linear regression $\hat{g}(A_i)=A_i x$ on the original dataset (for $i=1,\dots,N$), where $A_i$ denotes the $i$'th row of~$A$. We picked $d=100$, $N=1000$ and randomly generated the entries of $A$ from ${\cal N}(0,1)$. To generate the output $y_i$ for data point $A_i$, we first sampled a vector $w$ from the interior of the $d$-ball with radius $R=10$ and added noise: $y_i=A_i w+\xi_i$, where $\xi\sim{\cal N}(0,R^2/16)$ is the noise. We determined the optimal solution $x^\star$ using the closed-form solution $    x^\star=(A^TA)^{-1}A^Ty$.

The problem is strongly convex with condition number $\kappa={\sigma_{\max}(A)}/{\sigma_{\min}(A)}$, where $\sigma_{\max}(A)$ and $\sigma_{\min}(A)$ are the maximal and the minimal singular values of $A$, respectively. We set $p=({1}/{2})({1}/({1+\kappa}))$ and $q=({3}/{4})({1}/({1+\kappa}))$. At each iteration, we picked the corrupt gradient as $b_t=({R}/{\gamma_t})(({x^\star-x_t})/{\|x^\star-x_t\|})$, which can be shown to be worst possible adversarial value at iteration $t$. 

In Figure~\ref{fig:linearregression} we compare the performances of several first-order optimization algorithms using mean squared error $\|x_t-x^\star\|^2$ as the performance metric. The first-order algorithms such as Gradient Descent (GD), ADAM \cite{kingma2014adam}, Nesterov's Accelerated Gradient Method (NAG) \cite{nesterov1983method}, RMSprop \cite{tieleman2012lecture}, and AMSGrad \cite{reddi2019convergence} fail to converge; whereas the normalized subgradient method (NSM) outlined in Algorithm~\ref{alg:normalizedgd} succeeds. Algorithms that do not normalize the gradient fail to converge due to the presence of the adversarial perturbations in the gradient. 


\subsection{Logistic Classification}

\begin{figure}[t]
    \centering
    \includegraphics[width=.5\textwidth]{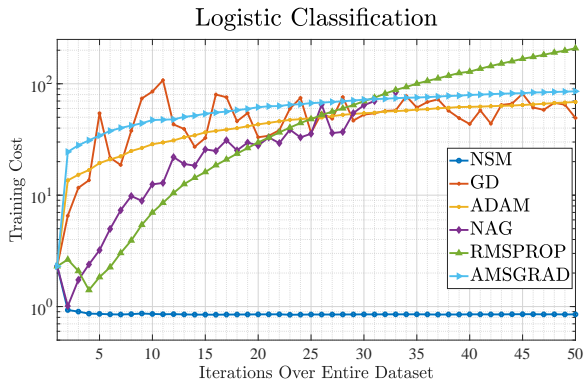}
    \vspace{-.5cm}
    \caption{Performances of first-order algorithms for Logistic Classification using training cost as metric.}
    \label{fig:logisticclassification}
\end{figure}

We study the robustness of Algorithm~\ref{alg:normalizedgd} on $L_2$-regularized multi-class logistic classification using the MNIST dataset~\cite{lecun1998gradient}. The task is to determine $m=10$ linear classifiers in order to separate $N=6000$ $d=784$-dimensional image vectors. The problem can be stated as:
\begin{equation}
    \underset{x\in\mathbb{R}^{m\times d}}{\min}-\frac{1}{N}\sum_{i=1}^N \log{\frac{e^{x_{y_i}A_i^T}}{\sum_{j=1}^m e^{x_{y_j}A_j^T}}}+\lambda\|x\|^2,
\end{equation}
where $A\in\mathbb{R}^{N\times d}$ is the matrix containing $N$ data vectors in its rows ($A_i$ denotes the $i$'th row of $A$) and $y\in \{0,1,\dots,9\}^N$ is the vector containing the $N$ associated classes. The decision parameter consists of $m$ vectors with dimension $d$, where each vector corresponds to a class (hence, $x_{y_i}$ corresponds to the $y_i$'th row of $x$, where $y_i$ is the class of $i$'th data vector). We note that the problem is unconstrained since ${\cal X}=\mathbb{R}^{m\times d}$ and therefore the parameter can go unbounded. Furthermore, although the problem is convex, it is not trivial to determine whether the cost function satisfies the acute angle condition. Nevertheless, we numerically experiment whether normalized subgradient method can provide convergence  in an adversarial setting.

We set $p=0.25$, $\gamma_t={0.1}/{t}$, and $\lambda=100$. At each iteration, we picked the corrupt gradient as $b_t=-15g(x_t)$, \ie 15 times the opposite of the trustworthy gradient. In Figure~\ref{fig:logisticclassification} we compare the performances of several first order optimization algorithms using the value of the cost function $f(x_t)$ during training as metric. Similar to Linear Regression experiment, NSM converges while other algorithms diverge.

\subsection{Threshold Probability}\label{sec:prob}

In this experiment, we demonstrate the significance of the threshold probability, for which our theoretical convergence results hold, using the toy example outlined in Section~\ref{sec:problem}. Note that for the cost function $f(x)=x_{d-1}^4$, the acute angle condition holds with equality for $\cos{\phi}=1/\sqrt{d}$. Therefore, the probability ${1}/({1+\sqrt{d}})$ is the exact threshold below which NSM converges and above which the adversary can cause NSM to diverge. We test this numerically for $d=10$, $R=10$, and $p_{\textnormal{threshold}}={1}/({1+\sqrt{d}})=0.2403$. We initialized $x_{0,i}=5,~\forall i \in [1,\cdots,d]$, and ran NSM for $p=\{0.1,0.2,p_{\textnormal{threshold}}-0.01,p_{\textnormal{threshold}},p_{\textnormal{threshold}}+0.01,0.4\}$ with $\gamma_t=200/t$. At each iteration, we set $b_t=-x_t$ if $x_{t,0}\neq 0$; and $b_t=[1,\dots,1]^T$ otherwise. In Figure~\ref{fig:simpleexample}, we plot $x_{t,0}$ versus the iteration index $t$ for each corruption probability. We observe that for $p<p_{\textnormal{threshold}}$, the iterate converges to the optimal solution and for $p>p_{\textnormal{threshold}}$, the iterate diverges (saturates at the bound). For $p=p_{\textnormal{threshold}}$, the iterate does not diverge, however, it converges to a suboptimal solution.

\begin{figure}[t]
    \centering
    \includegraphics[width=.5\textwidth]{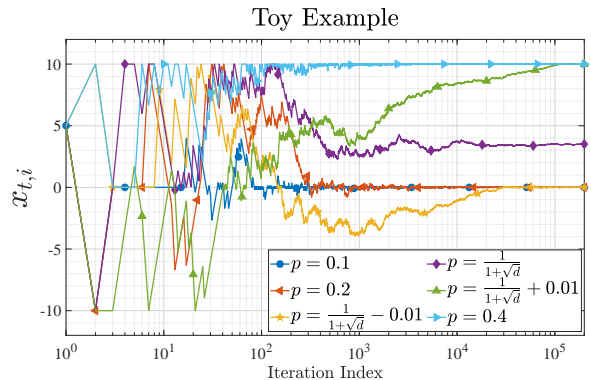}
    \vspace{-.5cm}
    \caption{Iterates generated by NSM for the simple example outlined in Section~\ref{sec:problem} for several corruption probabilities.}
    \label{fig:simpleexample}
\end{figure}

\section{CONCLUSIONS}
In this paper, we investigated the robustness of the normalized subgradient method in a randomly corrupted subgradient setting, where the subgradient at each iteration is arbitrarily corrupted with probability $p$. We have shown that if the cost function satisfies the acute angle condition for some angle $\phi$ and if $p$ is less than ${\cos{\phi}}/({1+\cos{\phi}})$, then the normalized subgradient method, with stepsizes that diminish as ${\cal O}(1/t)$, converges almost surely to a minimizer of the objective function. We provide sufficient conditions for convergence of the normalized (sub)gradient method for convex, strongly convex, and weakly pseudo-convex cost functions. Numerical examples show the robustness of the normalized subgradient method for linear regression and logistic classification problems.


\bibliographystyle{IEEEtran}
\bibliography{references, robust_federated, ngd, adversarial_corruption}


\end{document}